\documentclass[12pt]{article}

\usepackage{amsmath, amsthm, amssymb, amsfonts, graphicx}
\usepackage{comment}
\usepackage[margin= 1.125 in]{geometry}
\usepackage{xcolor}

\theoremstyle{definition}
\newtheorem{thm}{Theorem}[section]
\newtheorem{lem}[thm]{Lemma}
\newtheorem{prop}[thm]{Proposition}

\theoremstyle{definition}
\newtheorem{defn}{Definition}[section]

\theoremstyle{remark}

\numberwithin{equation}{section}

\def\N{{\mathbb N}}

\def\R{{\mathbb R}}

\def\Z{{\mathbb Z}}

\def\0{{\mathbf 0}}
\def\1{{\mathbf 1}}

\def\Ocal{{\mathcal O}}

\def\sup{\mathrm{sup}}

\def\<{\langle}
\def\>{\rangle}

\DeclareMathOperator{\qand}{\;\;\text{and} \;\;}

\allowdisplaybreaks

\begin{document}
\renewcommand{\epsilon}{\varepsilon}

\title{Statistics of a Family of Piecewise Linear Maps}
\author{J. J. P. Veerman\thanks{Fariborz Maseeh Dept. of Math. and Stat., Portland State Univ., Portland, OR, USA; e-mail: veerman@pdx.edu}, P. J. Oberly\thanks{Dept. of Math., Oregon State University; e-mail: oberlyp@oregonstate.edu}, L. S. Fox\thanks{Fariborz Maseeh Dept. of Math. and Stat., Portland State Univ., Portland, OR, USA; e-mail: logfox@pdx.edu}\\
}\maketitle
\date{}

\begin{abstract}
We study statistical properties of a family of piecewise linear monotone circle maps $f_t(x)$ related
to the angle doubling map $x\rightarrow 2x$ mod 1.
In particular, we investigate whether for large $n$, the \emph{deviations}
$\sum_{i=0}^{n-1} \left(f_t^i(x_0)-\frac 12\right)$
upon rescaling satisfy a $Q$-Gaussian distribution if $x_0$ and $t$ are both independently and uniformly
distributed on the unit circle. This was motivated by the fact that if $f_t$ is the rotation
by $t$, then it was recently found that in this case the rescaled deviations are distributed as a $Q$-Gaussian
with $Q=2$ (a Cauchy distribution). This is the only case where a non-trivial (i.e. $Q\neq 1$) $Q$-Gaussian
has been analytically established in a conservative dynamical system.

In this note, however, we prove that for the family considered here, $\lim_n S_n/n$ converges to a
random variable with a curious distribution which is clearly not a $Q$-Gaussian or any other standard
smooth distribution.
\end{abstract}

\begin{section}{Introduction}                        

In this note, we study statistical properties of the family of truncated flat spot maps $f_t$, defined
as follows. One starts with $x\mapsto 2 x$ and truncates horizontally to obtain
a ``flat spot circle map" as illustrated in Figure \ref{fig:flatspotmap} and given by
\begin{equation*}
f_t(x)=\left\{\begin{matrix} t &  x\in [0,\frac t2]\\[0.1cm]
           2x \mod 1 & x \in [\frac{t}{2},\frac{1+t}{2}] \\[0.1cm]
          t & x \in [\frac{1+t}{2}, 1) \end{matrix} \right. .
\end{equation*}
In particular, we wish to study the quantity
\begin{equation*}
 S_n(t,x_0) := \sum_{i=0}^{n-1} \left(f_t^i(x_0)-\frac 12\right) .
\end{equation*}
By this we mean the histogram of $S_n(t,x_0)$ \emph{while keeping $n$ fixed and varying $t$ and $x_0$}.
We are interested in the limiting behavior of $S_n$ as $n$ tends to infinity.
Thus, in computations using $S_n$, one requires that $n$ is
``large". It is  easy to generalize all results of this paper to truncations of
$x\mapsto \tau x$ for any $\tau\in \Z$ with modulus greater than 1 (see \cite[Section 5]{Vsymbolic}
for some of the details). For simplicity, however,
we will stay with $\tau=2$ in this note. We will refer to the distribution of $S_n$ as the
distribution of the \emph{deviations} of the family $f_t$. The question we want to investigate is
whether the distribution of the values of $S_n$ upon rescaling satisfy a nontrivial $Q$-Gaussian
distribution (see Definitions \ref{def:functions} and \ref{def:stats} at the end of this section).

\vskip -0.0in
\begin{figure}[h]
\begin{center}
\includegraphics[height=4.0cm]{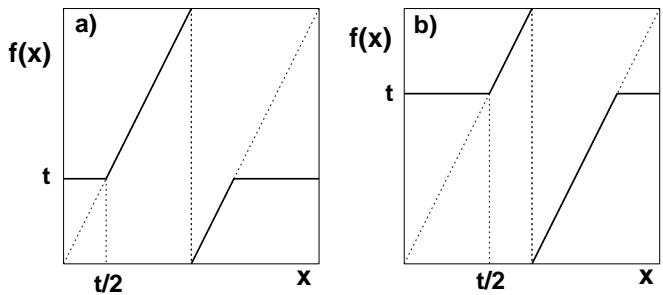}
\caption{The construction of $f_t(x)$, a) for $t=1/3$ and b) for $t=2/3$.}
\label{fig:flatspotmap}
\end{center}
\end{figure}

In the 1980s, C. Tsallis and others suggested (\cite{Tsallis09} and references therein) that in many physical
systems, in particular those with long range interactions, the values of certain observables might have
distributions that are not Gaussian, but a generalization thereof called $Q$-Gaussian.
In particular, this question has been studied for low-dimensional systems such as the standard map from
the unit torus $S^1\times S^1$ to itself
\begin{equation*}
f(x,y)=\left( x+y-k\sin(2\pi x), y-k\sin(2\pi x) \right) .
\end{equation*}
Here one fixes $n$, and chooses initial conditions $(x,y)$ uniformly distributed in $S^1\times S^1$.
In that case, for large $k$ (e.g. \(k=10\)), the dynamics tends to \emph{mostly} have large Lyapunov
exponents and correlations die out quickly. Thus, in this case, one expects the distribution of $S_n$ to be
\emph{approximately} the standard Boltzmann-Gibbs (or Gaussian) one, which coincides with the $Q$-Gaussian distribution for $Q=1$. (The fact that it is not exact is due the presence of elliptic islands even
for large $k$, see for example \cite{Karney, Duarte}.) Numerically this has been verified \cite{Chi, TB}.
For low $k$, on the other hand, correlations die out more slowly, and according to Tsallis' theory,
this would give a non-trivial $Q$-Gaussian distribution. Indeed for $k=0.2$, numerics show that that
is the case \cite{TB}. Somewhat surprisingly, in \cite{TT} it is argued that this curious behavior
appears to persist for $k=0$. Numerically, the authors measured $Q=1.935$ in that case.

Remarkably, even though the standard map with $k=0$ would seem to be excruciatingly simple (namely a
lamination of pure rotations), the analysis of its statistics is far from easy. In fact, its analysis
relies on a difficult theorem by Kesten \cite{Kesten}. Indeed, in this case \cite{BVV}, the distribution
of $S_n/\ln(n)$ tends to a Cauchy distribution, which again, as luck would have it, is a special case
of the $Q$-Gaussian distributions, namely $Q=2$ (not 1.935 as the measurements in \cite{TT} 
seemed to suggest).

To date this seems to be the only \emph{conservative} dynamical system where analytic proof
can be given that the deviations $S_n$ follow a non-Gaussian distribution. In this case, the
(rescaled) deviations tend to a non-trivial (i.e. $Q\ne 1$) $Q$-Gaussian distribution.
(For overdamped many body systems, there are various results known, see \cite{Moreira-ea-2018} and
references therein.) Because the topic of the deviations has attracted both interest and
controversy, any other examples where deviations can be analyzed rigorously are very valuable.
Below, we present such an example. It is intriguing that while the above example conforms to the theory
proposed by Tsallis, the example that we will now discuss clearly does not.

This brings us back to the the family $f_t$. Just like the $k=0$ standard map, it is a family of (weakly)
monotone circle maps with rotation numbers varying between 0 and 1.
Furthermore, like the low $k$ standard map, its dynamics is somewhere between chaotic and
completely simple. For any fixed $t$, the dynamics on any invariant set is semi-conjugate to a
rotation, and thus definitely not chaotic. But on the other hand, the dynamics is not quite as tame as
the family of rotations that constitute the standard map with $k=0$. Each individual map $f_t$ has
a non-trivial attractor and repeller. Thirdly, the dynamics of this
family is intimately related to that of the large $k$ standard map (see \cite{VTreno}). One can show
that the geometry of the repelling invariant sets of irrational rotation number is the same as the
asymptotic geometry of the Aubry Mather sets in the standard map \cite{TVasypmtgeom}. This becomes
clear if one locally `renormalizes' the standard map around these orbits \cite{VTreno}. In view of
the numerical results for the standard map just mentioned, a study of this family thus becomes
doubly interesting. The point here is not to define a 1-parameter family of maps that somehow mimics
the  low-$k$ standard map, but rather, to define a 1-parameter family of maps that shares
important characteristics (dynamics less tame than rotation, but still related to rotation)
with the low-$k$ standard map.

In Sections 2 and 3, we prove that the distribution $S_n/n$ converges to a fixed distribution.
We also obtain an exact expression for that distribution as an infinite sum.
In Section 4, we compute the error caused by approximating this distribution using a partial sum.
In Section 5, we study the tail of the distribution of $S_n/n$ and prove that it cannot be
(an approximation of) the tail of any $Q$-Gaussian.
In the concluding section, we exhibit approximations of the distribution and discuss its characteristics.
We conclude that, in spite of some superficial similarities, the distribution of $S_n/n$
is not a $Q$-Gaussian, neither can it be smoothed or averaged to give a $Q$-Gaussian.

For the record, we present some of the definitions relevant to our discussion here.
\begin{defn} For \(Q>0\), the \(Q\)-exponential $e_Q: \R \to (0,\infty)$, and its inverse, the
\(Q\)-logarithm $\ln_Q:(0,\infty)\to \R $, are given by
\[ e_Q(x)=(1+(1-Q)x)^{1/(1-Q)} \qquad \text{and} \qquad \ln_Q(y)=\frac{1-y^{Q-1}}{Q-1} . \]
\label{def:functions}
\end{defn}

It is straightforward to check that these functions are indeed inverses of one another and that
taking the limit \(Q\to 1\) gives the usual exponential and (natural) logarithm.

\begin{defn} We say that $y$ satisfies nontrivial $Q$-Gaussian statistics if the density of $y$ is given by
$C \, e_Q(-\beta (y-y_0)^2)$ for some $Q\ne 1$, where $C$ is a normalization constant.
\label{def:stats}
\end{defn}

\vskip .1in\noindent
{\bf Acknowledgements:} We are grateful to Tassos Bountis, Ugur Tirnakli, and Constantino Tsallis for
several useful conversations.

\end{section}

\begin{section}{Distribution of $\mathbf{S_{\mathbf n}/n}$}                    

We start with a result that summarizes some of the considerations in \cite{Vsymbolic,Virrat}. In what
follows, $\{x\}$ denotes the fractional part of $x$.

\begin{prop}[\cite{Vsymbolic,Virrat,Vhdim}]\label{prop:veermanpapers}
For each rational number $p/q \in (0,1)$ with \(\gcd(p,q) = 1\), there is a unique interval
$I_{p/q} \subset [0,1]$ such that: \\
(i) For all $t\in \textrm{int}( I_{p/q})$, the maps \(f_t\) have a unique \(q\)-periodic \emph{unstable} orbit $\Ocal_{p/q}^u$ and a unique \(q\)-periodic \emph{stable} orbit $\Ocal_{p/q}=(t,\{2t\},\dots , \{2^{q-1}t\}, t ,\dots)$. \\
(ii) All $x \in [0,1]\setminus \Ocal_{p/q}^u$ eventually reach the stable periodic orbit. \\
(iii) The set $[0,1]\setminus \bigcup_{p/q \in (0,1)} I_{p/q}$ has Hausdorff dimension $0$. \\
We will denote intervals $I_{p/q}$ by \emph{$p/q$ intervals}.
\end{prop}

We remark that if $t$ is in the boundary of an interval $I_{p/q}$, then the stable and unstable
orbits coincide (becoming stable in one direction and unstable in the other). For irrational values in \((0,1)\),
something similar happens, and there is a unique invariant set, which is stable in one direction
and unstable in the other. The details of the correct description are quite cumbersome, and not necessary for this exposition,
so we leave them out.

For the remainder of this section, denote the flat spot $[0,\frac{t}{2}]\cup [\frac{1+t}{2},1)$ by $F_1$.
Restricted to the closure of its complement $\overline{F_1^c}$ (see Figure \ref{fig:flatspotmap}),
$f_t: \overline{F_1^c} \rightarrow [0,1]$ is a continuous bijection. Thus, for every positive $i$, there
is a unique non-empty interval that is the inverse image $f_t^{-i}(F_1)$ of $F_1$. Denote this inverse
image by $F_{i+1}$. By doing so, we have \(f_t^i(F_i) = t\).

%
%
%

\begin{thm}\label{thm:SnLimit}
Fix \(p/q\) with \(\gcd(p,q)=1\). For all \(t\in \textrm{int}(I_{p/q})\) we have
\[ \lim_{n\to \infty} \frac{S_n(t,x_0)}{n} = \frac{1}{q} \sum_{i=0}^{q-1} \left( \{2^i t\} - \frac{1}{2} \right) \]
for almost all \(x_0\in [0,1]\).
\end{thm}

\begin{proof}
Fix \(t\in \textrm{int}(I_{p/q})\) and let \(x_0\) be given such that \(x_0\notin \Ocal_{p/q}^u\).
By item (ii) of Proposition \ref{prop:veermanpapers} and the definition of the sets \(F_i\), we can find some \(\ell\in \N\) such that \(x_0 \in F_\ell\). For any \(n\in \N\) sufficiently large, we can write \(n = \ell + kq + j\) where \(1\leq j\leq q-1\). Using this substitution for \(n\), and noting that \(f^i_t(x)\) enters the stable orbit \( t , \{2t\} , \cdots , \{2^{q-1} t\} , t , \cdots \) beginning at \(i=\ell\), we can write
\begin{align*}
S_n(t,x_0) & = \sum_{i=0}^{n-1} \left( f_t^i(x_0) - \frac{1}{2} \right)
\\ & = \sum_{i=0}^{\ell -1} \left( f_t^i(x_0) - \frac{1}{2} \right) + k\sum_{i=0}^{q-1} \left( \{2^i t\} - \frac{1}{2} \right) + \sum_{i=0}^{j-1} \left( \{2^i t\} - \frac{1}{2} \right)
\\ & = S_{\ell}(t,x_0) + \sum_{i=0}^{j-1} \left( \{2^i t\} - \frac{1}{2} \right) + k\sum_{i=0}^{q-1} \left( \{2^i t\} - \frac{1}{2} \right) .
\end{align*}
Note that \(\left| S_{\ell}(t,x_0) + \sum_{i=0}^{j-1} \left( \{2^i t\} - \frac{1}{2} \right) \right| \leq \frac{\ell + q}{2}\).
Now, again using \(n=\ell+kq+j\), we observe that
\begin{align*}
& \left| \frac{S_n(t,x_0)}{n} - \frac{1}{q} \sum_{i=0}^{q-1} \left( \{2^i t\} - \frac{1}{2} \right) \right|
\\ & \quad = \left| \frac{1}{n} S_{\ell}(t,x_0) + \frac{1}{n}\sum_{i=0}^{j-1} \left( \{2^i t\} - \frac{1}{2} \right) - \frac{\ell + j}{q(\ell + kq + j)} \sum_{i=0}^{q-1} \left( \{2^i t\} - \frac{1}{2} \right) \right|
\\ & \quad \leq \frac{\ell + q}{2n} + \frac{\ell + j}{qn} \sum_{i=0}^{q-1} \left( \{2^i t\} - \frac{1}{2} \right)
\end{align*}
Taking the limit as \(n\to \infty\) yields the result.
\end{proof} 

Notice that the initial starting point $x_0$ has little influence on the distribution of $S_n/n$ for
large $n$. Thus to study the distribution $S_n(t, x_0)/n$, it suffices to consider the contribution by
$t$ alone. By item (iii) of Proposition \ref{prop:veermanpapers}, this can be achieved by considering $t \in I_{p/q}$
and then summing over all rationals to obtain the full distribution.
We fix some notation. Let $\nu$ be the density of the distribution of $S_n(t, x)/n$ as $n$ tends to infinity,
with $t$ and $x$ uniformly distributed in $[0,1]$, and let $\nu_{p/q}$ be the density of the distribution
of $S_n(t, x)/n$ as $n$ tends to infinity when $t \in I_{p/q}$. Then we have pointwise convergence:
\begin{equation}
\nu(z) = \sum_{q = 2}^\infty \sum_{\substack{ 1 \leq p < q \\ \gcd(p,q) = 1}} \nu_{p/q}(z) .
\label{eq:nu(z)}
\end{equation}
We therefore turn our attention to the computation of $\nu_{p/q}$.
\end{section}

\begin{section}{Computing $\nu_{p/q}$}                                                                       

The first step is to compute the $p/q$ intervals $I_{p/q}$. From now on, we will
always assume that a rational $p/q$ is given in lowest terms, i.e. $\gcd(p,q)=1$.
Let $s^+(p/q) = s^+$ be the binary string defined \cite{Vsymbolic,Virrat} by
\begin{equation}\label{eq:stringdef2}
    s^+_i = \lfloor (i+1)p/q \rfloor - \lfloor ip/q\rfloor .
\end{equation}
Notice that $s^+(p/q)$ is periodic with period $q$, and that $s^+_q = 0$. Furthermore, the first
$q$ entries of $s^+$ contain exactly $p$ ``$1$''s. Finally, we need to define
\begin{equation}
t_{p/q}=\sum_{i=1}^{q}s_i^+2^{-i}\,.
\label{eq:t0}
\end{equation}

\begin{lem}[\cite{Vsymbolic,Virrat}]
Let $s^+$ be the binary string associated to the rotation number $p/q \in (0,1)$ as above.
The \(p/q\) interval is given by
\[ I_{p/q} = \left[ \frac{2^q t_{p/q}-1}{2^{q}-1} , \frac{2^q t_{p/q}}{2^{q}-1} \right] , \]
where \(t_{p/q}\) is as given in (\ref{eq:t0}).
\end{lem}

\begin{lem}\label{lem:equalendpoints}
For simplicity of notation, let \(t^-\) and \(t^+\) be such that \(I_{p/q} = [t^- , t^+]\). Then
\[ \sum_{i=0}^{q-1}\,\left\{2^it^-\right\} = \sum_{i=0}^{q-1}\,\left\{2^it^+\right\} . \]
\end{lem}

\begin{proof}
Recall that \(t_{p/q}\) is defined using the binary string \(s_1^+ s_2^+ \cdots s_q^+\) which has precisely \(p\) ``1''s and \(q-p\) ``0''s. Therefore, \(t^+\) is represented by the infinite \(q\)-periodic binary string
\[ 0.s_1^+ s_2^+ \cdots s_q^+s_1^+ s_2^+ \cdots . \]
Given that \(s_q^+ = 0\) and \(t^- = t^+ - \frac{1}{2^q-1}\), we know that \(t^-\) is also an infinite \(q\)-periodic binary string with \(p\) ``1''s and \(p-q\) ``0''s in the first \(q\) digits. Thus, the given sums (which are sums over a binary shift) are in fact equal.
\end{proof}

\begin{thm}\label{thm:nudensity}
The density $\nu_{p/q}$ is given by
\[ \nu_{p/q} = \frac{q}{2^q - 1} \1_{J_{p/q}},\]
where $\1_X$ is the characteristic function of $X$ and
\[
J_{p/q} = \left[ \frac{p-t_{p/q}}{q} - \frac{1}{2} , \frac{p- t_{p/q} +1}{q} - \frac{1}{2} \right] .
\]
\end{thm}

\vskip -0.0in
\begin{figure}[h]
\begin{center}
\includegraphics[height=4.0cm]{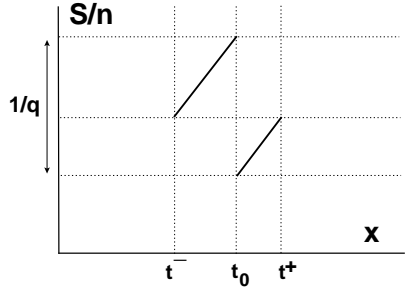}
\caption{Schematic picture of $\frac{S_n(t,x)}{n}$ for $t\in I_{p/q}=[t^-,t^+]$ as $n$ tends to infinity.}
\label{fig:discontinuity}
\end{center}
\end{figure}

\begin{proof}
To begin, we will restrict to $t\in \text{int}(I_{p/q})$. Thus $f_t$ must have a stable periodic orbit with a point
in the flat spot. We again use the simplified notation \(I_{p/q} = [t^- , t^+]\). 

By Theorem \ref{thm:SnLimit}, we have
\begin{equation}
\lim_{n\rightarrow \infty} \frac{S_n(t,x_0)}{n} = \frac{1}{q}\sum_{i=0}^{q-1} \left( \{2^i t\} - \frac{1}{2} \right).
\label{eq:Sovern}
\end{equation}
Denoting this limit by $\langle s\rangle$, we find
\begin{equation}
\partial_t\,\langle s\rangle = \frac{2^q-1}{q}\,,
\label{eq:derivative}
\end{equation}
except at the discontinuities of $\langle s\rangle$.

We know, however, exactly where those discontinuities arise.
For $t$ in the prescribed range, the $q$-tuple $(\{2^i t\})_{i=0}^{q-1}$ has precisely one discontinuity, namely for $i=q-1$.
Thus we have the situation where $\{2^{q-1}t^-\}$ equals the left endpoint of the flat spot while  $\{2^{q-1}t^+\}$ equals the right endpoint of the flat spot.
In between, at $t=t_{p/q}$ defined in equation \eqref{eq:t0}, we have
\begin{equation}
\lim_{t\nearrow t_{p/q}} \{2^{q-1}t\} =1 \quad \qand \quad \lim_{t\searrow t_{p/q}} \{2^{q-1}t\} =0\,.
\label{eq:discontinuity}
\end{equation}
Therefore, using \eqref{eq:Sovern}, we get (see Figure \ref{fig:discontinuity})
\[
\lim_{t\nearrow t_{p/q}} \langle s\rangle = \lim_{t\searrow t_{p/q}} \langle s\rangle - \frac 1q \,.
\]
Furthermore, this implies that the $q$th image of $t_{p/q}/2$ under $f_{t_{p/q}}$ equals 0.
Thus $t_{p/q}$ is the unique point of jump discontinuity.

Lemma \ref{lem:equalendpoints} also establishes that that at the ends of the interval $I_{p/q}$
\begin{equation}
\lim_{n\rightarrow \infty}\, \frac{S_n(t^-,x)}{n}= \lim_{n\rightarrow \infty}\, \frac{S_n(t^+,x)}{n}\,.
\label{eq:discont2}
\end{equation}
Statements \eqref{eq:derivative}, \eqref{eq:discontinuity}, and \eqref{eq:discont2} establish
the description of $\langle s\rangle$ as given in Figure \ref{fig:discontinuity}.

Thus for $t\in I_{p/q}$, the support $J_{p/q}$ of $\langle s\rangle$ is given by
\[
J_{p/q}=\left[\frac 1q\,\sum_{i=0}^{q-1}\,\left(\{2^i t_{p/q}\} - \frac 12 \right),
\frac 1q\,\sum_{i=0}^{q-1}\,\left(\{2^i t_{p/q}\} - \frac 12 \right)+\frac 1q\right]\,.
\]
By Proposition \ref{prop:geometricseries}, this is equal to
\[
J_{p/q} = \left[ \frac{p-t_{p/q}}{q} - \frac{1}{2} , \frac{p- t_{p/q} +1}{q} - \frac{1}{2} \right]\,.
\]

This interval is the support of $\nu_{p/q}$. For fixed $p/q$ and $t \in I_{p/q}$,
$\langle s \rangle$ is linear a.e. in $t$ (with one jump discontinuity). Therefore $\nu_{p/q}$ is a.e. equal to a
constant on its support. As $\int \nu_{p/q}(z)\; dz = |I_{p/q}| = (2^q - 1)^{-1}$, and $|J_{p/q}| = 1/q$,
this gives the constant $\frac{q}{2^q - 1}$.
\end{proof}

\vskip 0.0in\noindent
{\bf Remark:} Theorem \ref{thm:nudensity} implies that the support of $\nu_{p/q}$ is in $[-\frac 12, \frac 12]$.
In fact, it immediately follows that the support of the distribution of $\langle s\rangle$ equals
$[-\frac 12, \frac 12]$.

\end{section}

\begin{section}{Error Estimate for $\nu$}                                                                    


Let \(\nu_N\) be the \(N\)th partial sum of \(\nu\),
\[\nu_N(x) =\sum_{q = 2}^N \sum_{ \substack{ 1 \leq p < q \\ \gcd(p, q) = 1}} \nu_{p/q}(x) . \]
The primary result for this section is the absolute error of this partial sum, given in Theorem \ref{thm: sup norm error}.

\begin{lem}\label{lem:tripleint}
Fix an integer $q \geq 3$ and let $p_1$ and $p_2$ be distinct positive integers relatively prime to $q$.
If \(p_1< p_2\), then $J_{p_1/q} \cap J_{p_2/q} \ne \emptyset$ if and only if $p_2 = p_1 + 1$.
\end{lem}

\begin{proof}
Suppose $\tau \in J_{p_1/q} \cap J_{p_2/q}$. Let
$t_i = t_{p_i/q}$, where $t_{p/q}$ is defined in equation \eqref{eq:t0}. Then
\[ q\tau + q/2 \in [p_i-t_i, p_i - t_i + 1] \]
Note that $p_1-t_1<p_2-t_2$, so the two intervals overlap if and only if $p_2-t_2<p_1-t_1+1$.
Since
\[ p_2-t_2<p_1-t_1+1 \iff p_2 - p_1 < t_2 - t_1 + 1 \]
and $t_i \in (0, 1)$ with \(t_1 < t_2\), we must have \(p_2-p_1 = 1\).
\end{proof}

\begin{thm}\label{thm: sup norm error}
Fix an integer $N \geq 2$. Then
\begin{equation}\label{eq:supnormerror}
     \| \nu - \nu_N \|  < \frac{4(N+2)}{2^{N+1}-1}
\end{equation}
where \(\| f \| = \sup_{x \in [-1/2, 1/2]} |f(x)|\).
\end{thm}

\begin{proof}
Let $x \in [-1/2, 1/2]$. For any $q \geq 2$,
\[\sum_{ \substack{1 \leq p < q\\ \gcd(p, q) = 1}} \nu_{p/q}(x) \leq \frac{2q}{2^q - 1} \]
by Lemma \ref{lem:tripleint} and the definition of $\nu_{p/q}$.
Thus,
\[
0 \leq \nu(x) -  \nu_N(x) = \sum_{N+1}^\infty \sum_{ \substack{1 \leq p < q\\ \gcd(p, q) = 1}} \nu_{p/q}(x)
\leq \sum_{q = N+1}^\infty \frac{2q}{2^q - 1}\leq
2 \left( \frac{2^{N+1}}{2^{N+1} - 1} \right) \sum_{q = N+1}^\infty \frac{q}{2^q} .
\]
The last sum can be explicitly evaluated via
\[
\sum_{q = N+1}^\infty \frac{q}{2^q}=
-\partial_\alpha \left[ \sum_{q = N+1}^\infty e^{-q\alpha}\right]_{\alpha=\ln(2)} .
\]
We leave the details to the reader, but the upshot is
\[
0 \leq \nu(x) -  \nu_N(x) \leq 2 \left( \frac{2^{N+1}}{2^{N+1} - 1} \right) \left( \frac{N+2}{2^{N}} \right) = \frac{4(N+2)}{2^{N+1}-1} .
\qedhere \]
\end{proof}

\end{section}

\begin{section}{The Tail of the Distribution}

As we will see in the conclusion, the general appearance of the final distribution $\nu(x)$ is rather spiky.
One might still think that perhaps $\nu(x)$ could be some kind of a piecewise constant approximation of a $Q$-Gaussian (see next section). In order to definitively dispel that idea, we compare the tails of
both distributions: the $Q$-Gaussian and $\nu(x)$ and show that these exhibit a very different behavior.

Since the support of $\nu$ is bounded, we know that if it is a $Q$-Gaussian, then $Q<1$.
In fact, $\nu$ shows a superficial similarity with a $Q$-exponential for $Q$ near 0.7 (see next section).
The average of the distribution is clearly 0, so we can take $y_0=0$ in Definition \ref{def:stats}.
We now use Definitions \ref{def:functions} and \ref{def:stats}, to prove the following result.

\begin{prop} Let $-\ell$ be the left endpoint of the support of a $Q$-Gaussian density $e_Q$ with $Q<1$.
For small positive $z$, we have
\[
e_Q(-\ell+z)= K z^{\frac{1}{1-Q}} +{\cal O}(z^{\frac{2}{1-Q}})\,,
\]
for some $K>0$.
\label{prop:gauss-tail}
\end{prop}

\begin{proof} According to the definitions just cited, $e_Q(x)=C\left(1-(1-Q)\beta z^2\right)^{1/(1-Q)}$.
The support of this distribution is $[-\ell,\ell]$, where $\ell=\left[(1-Q)\beta\right]^{-1/2}$. So if we set
$x=-\ell+z$, we obtain
\[
e_Q(-\ell+z)=C\left[2\ell^{-1}z-\ell^{-2}z^2\right]^{1/(1-Q)} \,.
\]
The conclusion follows by taking $z$ to be small (and positive).
\end{proof}

\begin{prop} Take $z$ to be small and positive and equal to $1/q$ (with $q\in\N$). Now we have
\[
\nu\left(-\frac 12 +z\right)= 2\;\frac{2^{-1/z}}{z}+{\cal O}\left(\frac{2^{-2/z}}{z^2}\right)\,.
\]
\label{prop:nu-tail}
\end{prop}

\begin{proof} First we compute a lower bound for $\nu\left(-\frac 12 + \frac 1q\right)$ for a large fixed $q$.
According to equation \ref{eq:nu(z)} and Theorem \ref{thm:nudensity}, the final density $\nu$
contains all terms $\nu_{1/k}$ which has support
$J_{1/k}=\left[\frac{1-t_{1/k}}{k}-\frac 12, \frac{2-t_{1/k}}{k}-\frac 12\right]$, where $t_{1/k}=2^{1-k}$.
Fix $q$ and let $k>q$, then $(-\frac 12 + \frac 1q)$ is in $J_{1/k}$ iff $q\leq k$ and
\[
\frac{1}{q}<\frac{2-2^{1-k}}{k} \quad \Leftrightarrow \quad k<2q-q2^{1-k}\,.
\]
Thus the intervals $J_{1/k}$ overlap at $-\frac 12 +z$ at least when $k\in \{q, \cdots, 2q-1\}$ for $q$ large. So
for $z=1/q$:
\[
\nu\left(-\frac 12 +\frac 1q\right)\geq \sum_{k=q}^{2q-1}\,\frac{k}{2^k-1}\geq
\sum_{k=q}^{2q-1}\,\frac{k}{2^k}=2^{-q}(2q+2)-2^{-2q}(4q+2) \,.
\]
Now substitute $q=1/z$.

Next, we compute an upper bound. The reasoning is similar. Fix some large $q$. From the previous, we know that
the intervals $J_{1/k}$ overlap at $-\frac 12 +z$ when $k\in \{q, \cdots, 2q-1\}$. For larger values of $k$, Lemma
\ref{lem:tripleint} implies that for any fixed $k$ at most two of the $J_{r/k}$ overlap at $-\frac 12 +z$.
Thus
\[
\nu\left(-\frac 12 +\frac 1q\right) \leq \sum_{k=q}^{2q-1}\,\frac{k}{2^k-1}+ 2\; \sum_{k=2q}^{\infty}\,\frac{k}{2^k-1} \leq
\frac{2^q}{2^q-1}\,\sum_{k=q}^{2q-1}\,\frac{k}{2^k}+ 2\; \frac{2^{2q}}{2^{2q}-1}\,\sum_{k=2q}^{\infty}\,\frac{k}{2^k} \,,
\]
which, after some effort, simplifies to $2^{-q}(2q+2)+{\cal O}\left((2^{-q}(2q+2))^2\right)$.
Again, substitute $q=1/z$. Finally, put the two inequalities together to prove the proposition.
\end{proof}

The conclusion of this is that the two tails are clearly incompatible. The decay of $\nu$ is inconsistent
with any $Q$-Gaussian decay. In fact, it is a simple exercise to show that
$e_Q(-\ell+z)/\nu(-\frac 12 +z)$ tends to infinity as $z$ tends to 0 for any $Q<1$.

\end{section}

\begin{section}{Conclusion}

In the following table, as well as Figure \ref{fig: nu graphs}, we illustrate the way the density is generated
from the densities $\nu_{p/q}$ for $q\leq 5$. It is interesting to note that, indeed, whenever both $p$ and
$p+1$ are relative prime to $q$, then the support $\nu_{p/q}$ and that of $\nu_{(p+1)/q}$ overlap (see Lemma
\ref{lem:tripleint}). This phenomenon seems to be responsible a lot of the ``spiky-ness" of the
final distribution.

\begin{center}
\begin{tabular}{| c | c | c | c |}
\hline
\(p/q\) & \(t_{p/q}\) & \(I_{p/q}\) & \(J_{p/q}\) \\
\hline
\(1/2\) & \(0.10 = 1/2\) & \([ 1/3 , 2/3 ]\) & \([ -1/4 , 1/4 ]\) \\
\(1/3\) & \(0.010 = 1/4\) & \([ 1/7 , 2/7 ]\) & \([ -1/4 , 1/12 ]\) \\
\(2/3\) &  \(0.110 = 3/4\) & \([ 5/7 , 6/7 ]\) & \([ -1/12 , 1/4 ]\) \\
\(1/4\) & \(0.0010 = 1/8\) & \([ 1/15 , 2/15 ]\) & \([ -9/32 , -1/32 ]\) \\
\(3/4\) & \(0.1110 = 7/8\) & \([ 13/15 , 14/15 ]\) & \([ 1/32 , 9/32 ]\) \\
\(1/5\) &  \(0.00010 = 1/16\) & \([ 1/31 , 2/31 ]\) & \([ -25/80 , -9/80 ]\) \\
\(2/5\) & \(0.01010 = 5/16\) & \([ 9/31 , 10/31 ]\) & \([ -13/80 , 3/80 ]\) \\
\(3/5\) & \(0.10110 = 11/16\) & \([ 21/31 , 22/31 ]\) & \([ -3/80 , 13/80 ]\) \\
\(4/5\) & \(0.11110 = 15/16\) & \([ 29/31 , 30/31 ]\) & \([ 9/80 , 25/80 ]\) \\
\hline
\end{tabular}
\end{center}

\begin{figure}[!ht] 
\begin{center}
\includegraphics[width = 0.4\linewidth]{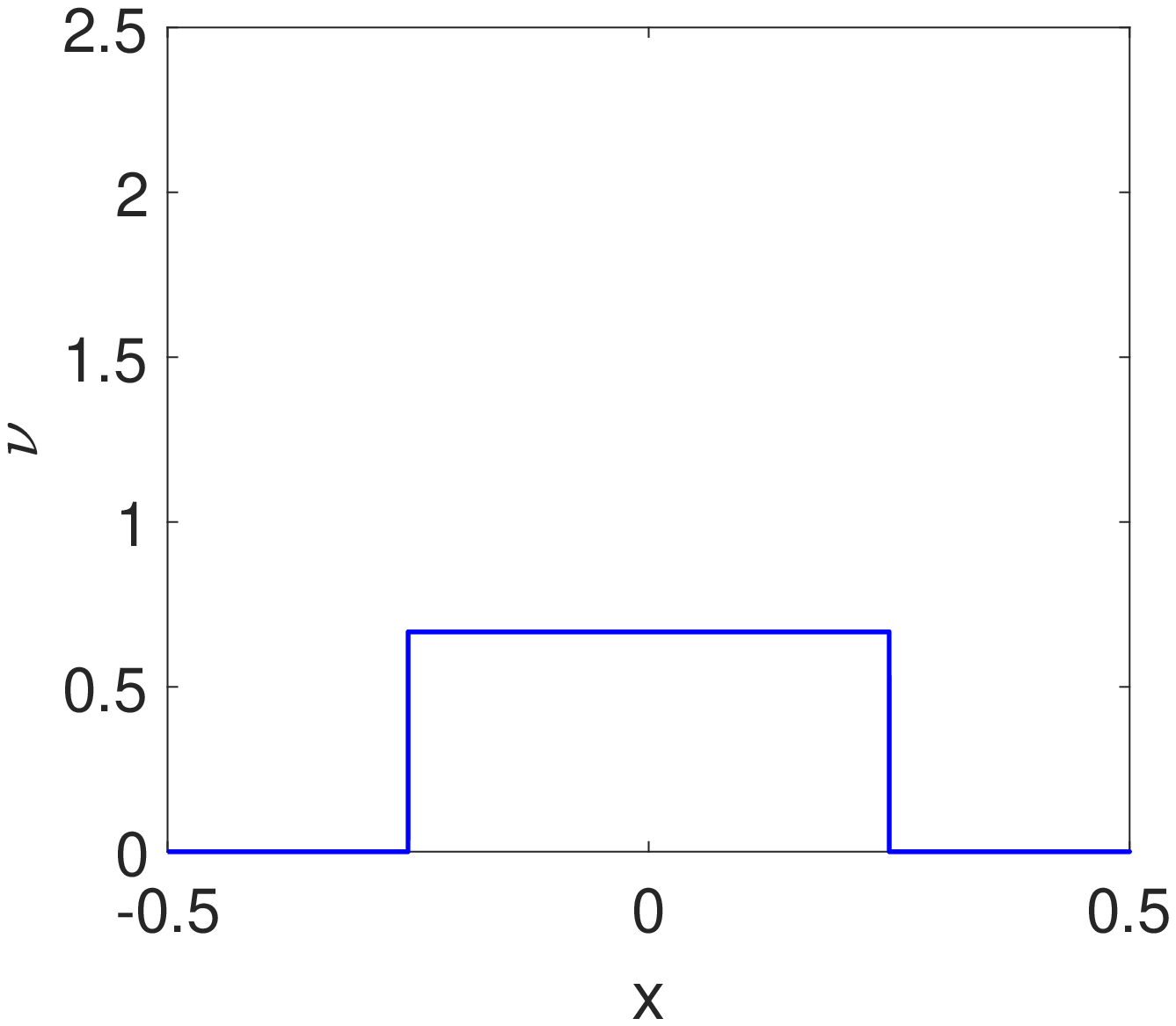} \includegraphics[width = 0.4\linewidth]{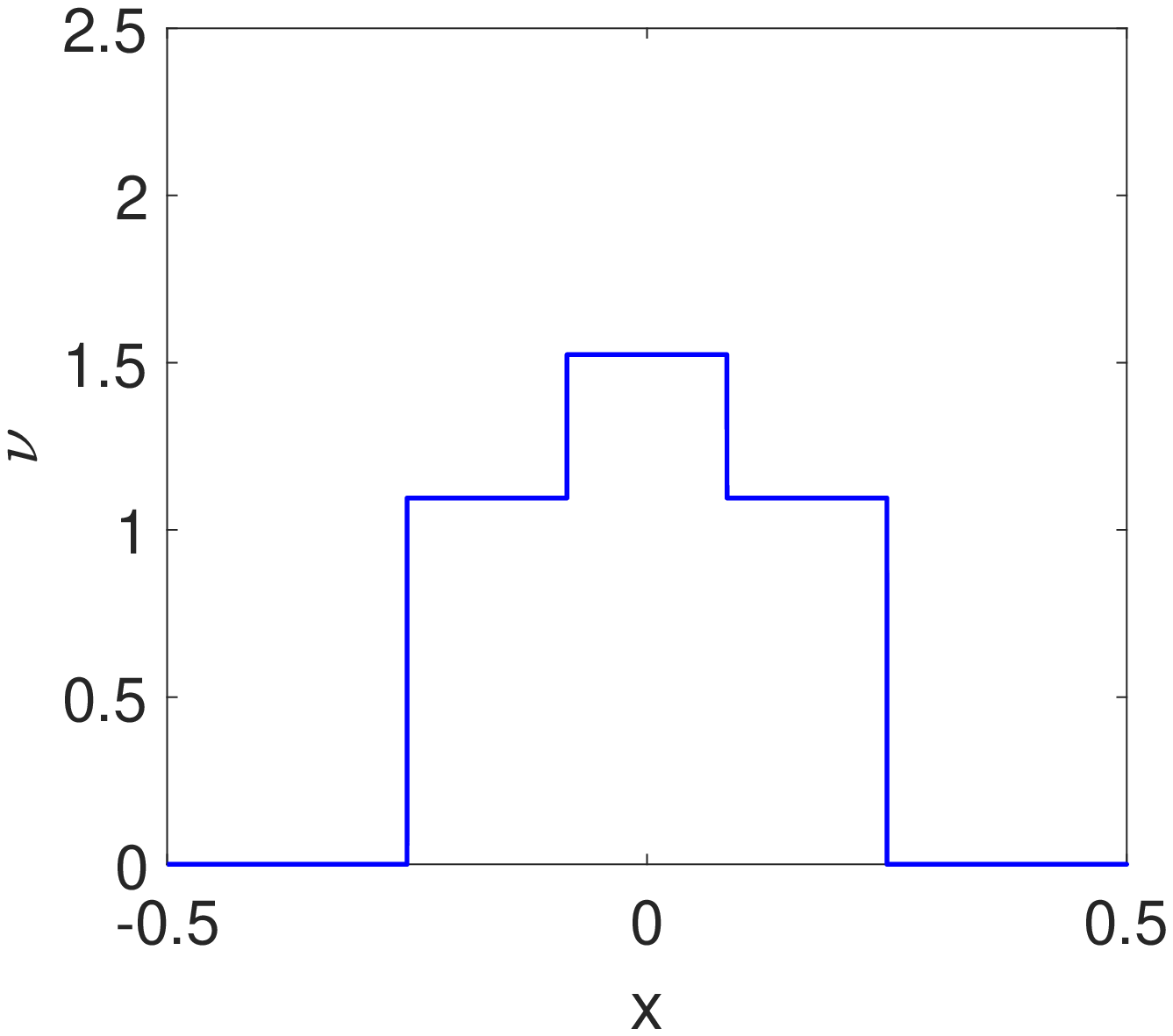} \\
\includegraphics[width = 0.4\linewidth]{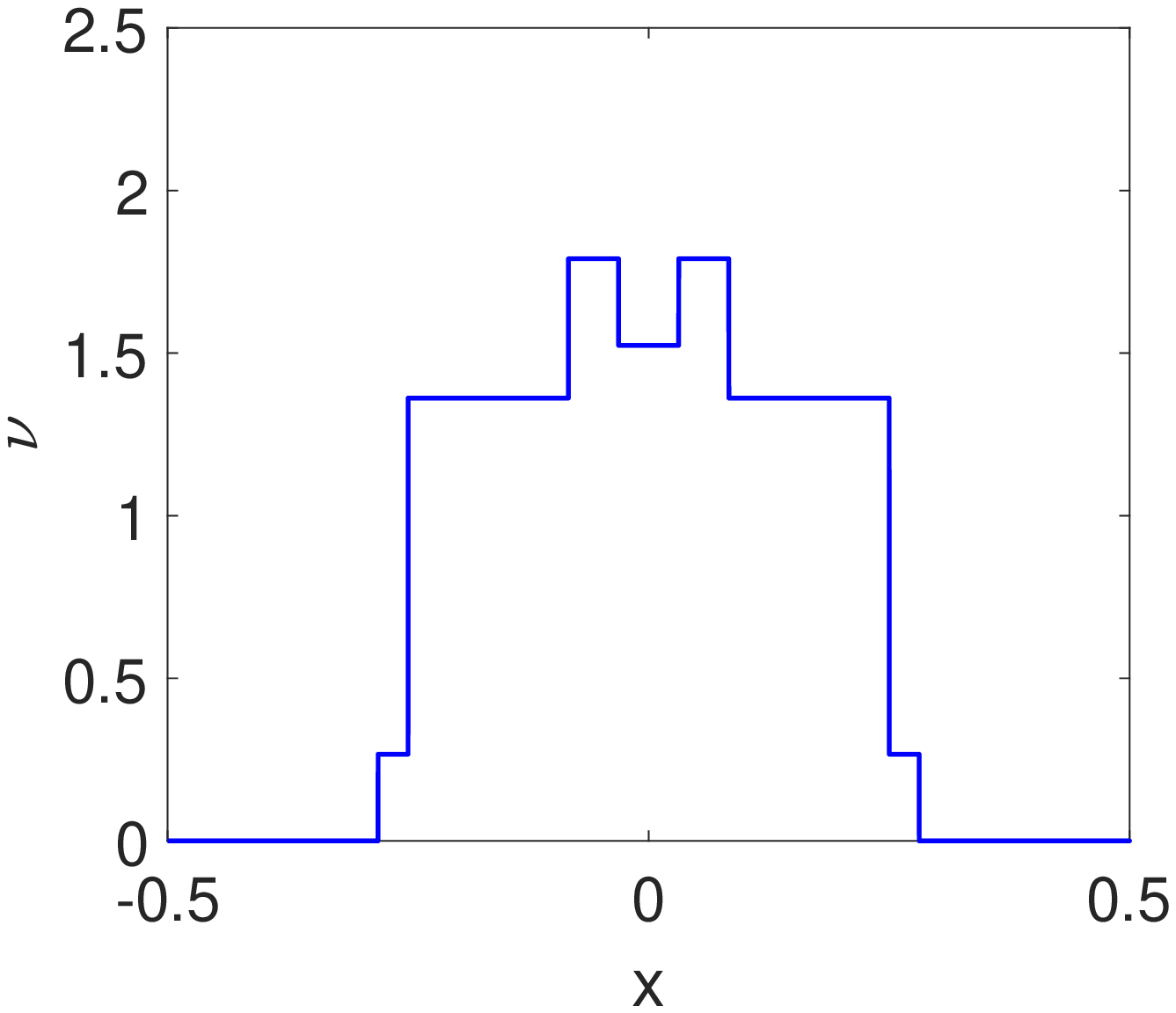} \includegraphics[width = 0.4\linewidth]{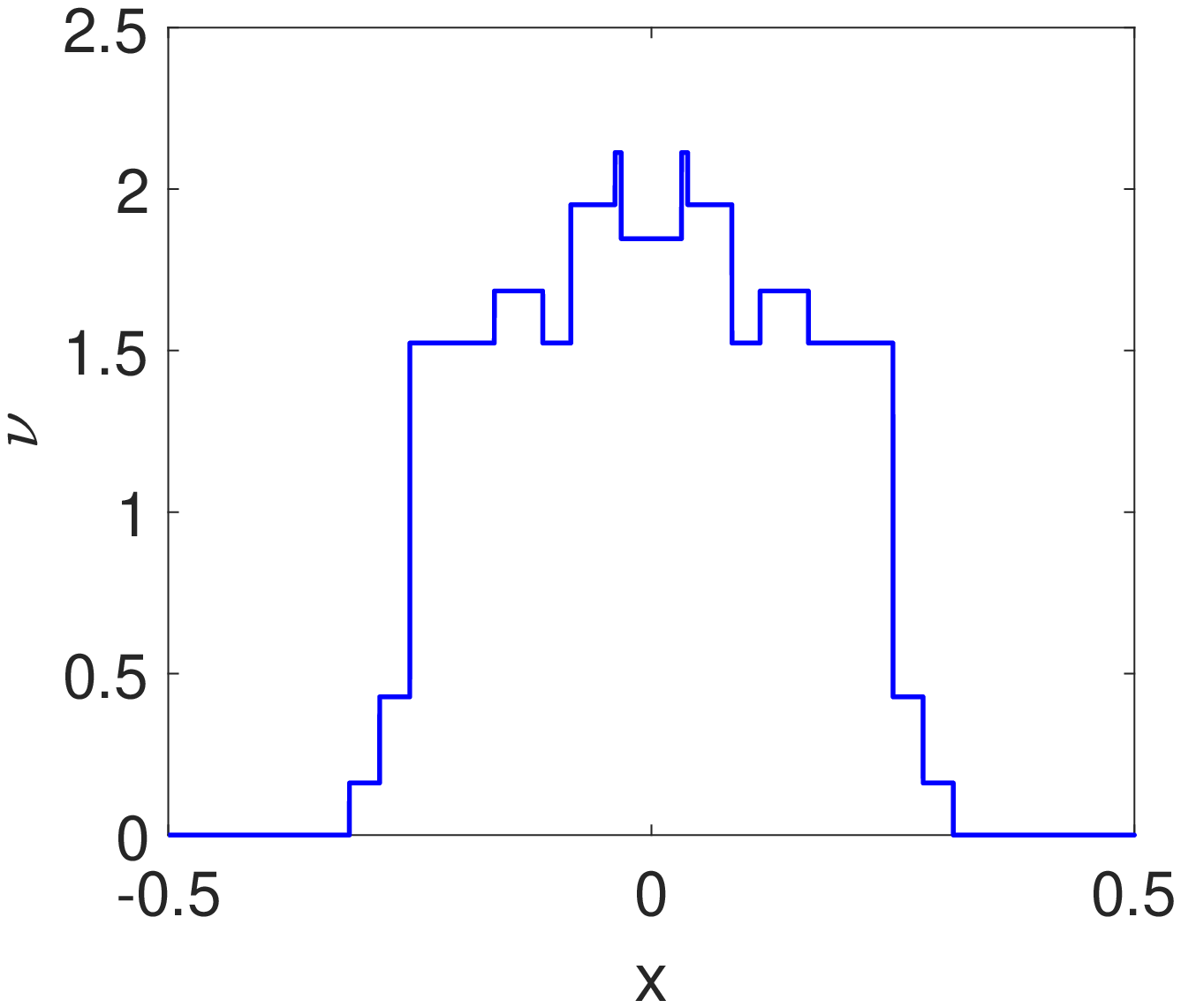}
\caption{From left to right, top to bottom: the partial sums \(\nu_2\), \(\nu_3\), \(\nu_4\), and \(\nu_5\).}
\label{fig: nu graphs}
\end{center}
\end{figure}

\begin{figure}[!ht] 
\begin{center}
\includegraphics[width = 0.55\linewidth]{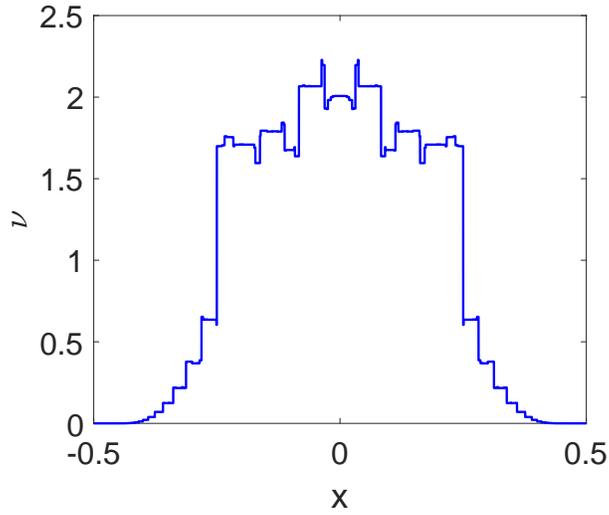}
\caption{The sum of all $\nu_{p/q}$ for $q\leq 50$.}
\label{fig:q50}
\end{center}
\end{figure}

\begin{figure}[!ht] 
\begin{center}
\begin{tabular}{ll}
\includegraphics[scale = 0.6]{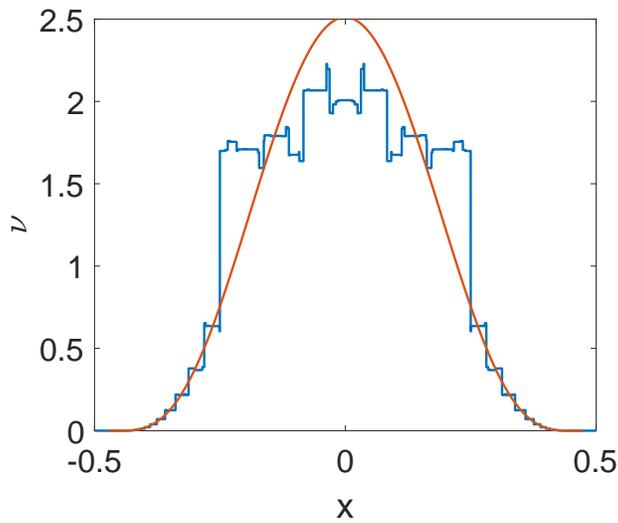}
\end{tabular}
\caption{The distribution \(\nu\) with the \(Q\)-exponential fit ($Q=0.7$, $\beta=16.1$).}
\label{fig: tsallis graph}
\end{center}
\end{figure}

Figure \ref{fig:q50} gives the partial sum \(\nu_{50}\). Theorem \ref{thm: sup norm error} tells us that
this figure approximates the true density $\nu$ with an (absolute) error less than $10^{-13}$.
The graph of this density resembles the profile of the building in the ``Ghostbusters" movie much more than it resembles a $Q$-Gaussian density for any $Q$ (see Definition \ref{def:stats}).

In Figure \ref{fig:q50}, it can be seen that, as mentioned before, that the $Q=0.7$-Gaussian traces the
average behavior of the data reasonably well, though it should be observed that the displayed $Q$-exponential is not normalized. In view of the fact that in earlier numerical simulations
$Q<1$ had already been observed in (critical) circle maps \cite{AT-2010}, one might thus be tempted
to conclude that a smoothed version this density could actually be $Q$-Gaussian. However, this notion
was dispelled by studying the tails of these distributions more carefully (see Section 5).

\end{section}

\appendix
\section{Summing fractional parts of a geometric series}

\begin{lem} Denote the expression $d_j2^{-j}$ with $d_{j}$ equal to 0 or 1 by $t_{-j}$.
For $1\leq j\leq n+1$:
\[
\sum_{i=0}^n\,\left\{2^it_{-j} \right\}={d_j-t_{-j}}\,.
\]
\label{lem:geometricseries}
\end{lem}

\begin{proof} For $i\in\{0,\cdots,j-1\}$, we have $\left\{2^it_{-j} \right\}= 2^it_{-j}$.
After that, $\left\{2^it_{-j} \right\}$ equals zero. Thus the sum in the lemma equals:
\[
\sum_{i=0}^{j-1}\,2^it_{-j} = \frac{2^j-1}{2 -1}\,d_j 2^{-j}= {d_j-d_j2^{-j}}\,,
\]
which equals the expression in the RHS of the lemma.
\end{proof}

\begin{prop} Now let $t:=\sum_{j=1}^q\,d_j2^{-j}=\sum_{j=1}^q\,t_{-j}$. Then
\[
\sum_{i=0}^{q-1}\,\left\{2^i t \right\}=\sum_{j=1}^q d_j-t
\quad \qand \quad
\sum_{i=0}^{q-1}\,\left\lfloor2^i t \right\rfloor=2^q t-\sum_{j=1}^q d_j\,.
\]
\label{prop:geometricseries}
\end{prop}

\begin{proof} That the formulae in the proposition are equivalent can be seen by summing them,
which yields the usual geometric sum. So it suffices to prove the first formula.

We have
\[
\sum_{i=0}^{q-1}\,\left\{2^i t \right\}=\sum_{i=0}^{q-1}\,\left\{2^i \sum_{j=1}^q\,t_{-j} \right\}
\]
Furthermore,
\[
\left\{2^i \sum_{j=1}^q\,t_{-j} \right\}=
\left\{2^i \sum_{j=1}^i\,t_{-j} + 2^i\sum_{j=i+1}^q\,t_{-j}\right\}\,
\]
and $2^i \sum_{j=1}^i\,t_{-j}$ is an integer while $2^i\sum_{j=i+1}^q\,t_{-j}$ is in $[0,1)$.
Thus Lemma \ref{lem:geometricseries} then gives
\[
\sum_{i=0}^{q-1}\,\sum_{j=1}^q\,\left\{2^i t_{-j} \right\} =
\sum_{j=1}^q\,d_j-t_{-j}\,,
\]
which in turns simplifies to the required expression.
\end{proof}

\end{document}